\documentclass [12 pt]{amsart}
\usepackage{amssymb,latexsym}
\usepackage{xypic}
\usepackage{csquotes}
\theoremstyle{plain}

\numberwithin{equation}{section}
\newtheorem*{theoA}{Theorem A}
\newtheorem*{theoB}{Theorem B}
\newtheorem*{theoC}{Theorem C}
\newtheorem*{theoD}{Theorem D}

\newtheorem{theo}{Theorem}[section]
\newtheorem{lem}{Lemma}[section]
\newtheorem{cor}{Corollary}[section]

\newtheorem{exm}{Example}[section]
\newtheorem{defi}{Definition}[section]
\newtheorem{rem}{Remark}[section]

\newcommand{\ol}{\overline}
\newcommand{\be}{\begin{equation}}
\newcommand{\ee}{\end{equation}}
\newcommand{\beas}{\begin{eqnarray*}}
\newcommand{\eeas}{\end{eqnarray*}}
\newcommand{\bea}{\begin{eqnarray}}
\newcommand{\eea}{\end{eqnarray}}

\theoremstyle{definition}

\theoremstyle{remark}

\AtBeginDocument{{\noindent\small Jordan Journal of Mathematics and Statistics (JJMS) 9(2), 2016, pp 117 - 139\\
DOI: }}

\begin{document}

\title[ Further results on the uniqueness of meromorphic functions ...\\
  ]{Further results on the uniqueness of meromorphic functions  and their derivative counterpart sharing one or two sets }
\date{}

\author[A. Banerjee and B. Chakraborty ]{ Abhijit Banerjee$^{1}$  and Bikash Chakraborty$^{2}$ }
\date{}

\address{$^{1}$ Department of Mathematics, University of Kalyani, West Bengal 741235, India.}
\email{abanerjee\_kal@yahoo.co.in, abanerjee\_kal@rediffmail.com
}
\address{$^{2}$ Department of Mathematics, University of Kalyani, West Bengal 741235, India.}
\email{bikashchakraborty.math@yahoo.com, bikashchakrabortyy@gmail.com}



\thanks{ The first author's research work is supported by the Council Of  Scientific and Industrial Research, Extramural Research Division, CSIR Complex, Pusa, New Delhi-110012, India, under the sanction project no. 25(0229)/14/EMR-II. \\
 The second author's research work is supported by the Department of Science and Technology, Govt. of India under the sanction order DST/INSPIRE Fellowship/2014/IF140903.}

\begin{abstract}
In this paper we prove a number of results concerning uniqueness of a meromorphic function as well as its derivative sharing one or two sets. In particular, we deal with the specific question raised in \cite{13.1}, \cite{13.2}, \cite{12.1} and ultimately improve the result of Banerjee-Bhattacharjee \cite{3.3}.
\end{abstract}
\maketitle
\section{Introduction and Definitions}
In this paper, we assume that readers familiar with the basic Nevanlinna theory(\cite{8}).
By $\mathbb{C}$ and $\mathbb{N}$ we mean the set of complex numbers and set of positive integers respectively. Let $f$ and $g$ be two non-constant meromorphic functions and let $a$ be a finite complex number. We say that $f$ and $g$ share the value $a$ CM (counting multiplicities), provided that $f-a$ and $g-a$ have the same zeros with the same multiplicities. Similarly, we say that $f$ and $g$ share the value $a$-IM (ignoring multiplicities), provided that $f-a$ and $g-a$ have the same set of zeros, where the multiplicities are not taken into account.
In addition we say that $f$ and $g$ share $\infty$ CM (IM), if $1/f$ and $1/g$ share $0$ CM (IM). \par
We now recall some well-known definitions in the literature of the uniqueness of meromorphic functions sharing sets as it will be pertinent with the follow up discussions.
\begin{defi}
For a non-constant meromorphic function $f$ and any set $S\subset \mathbb{C}\bigcup\{\infty\}$, we define
$$E_{f}(S)=\bigcup\limits_{a \in S}\{(z,p) \in \mathbb{C}\times\mathbb{N}~ |~ f(z)=a ~with~ multiplicity~ p\},$$
$$\ol{E}_{f}(S)=\bigcup\limits_{a \in S}\{z \in \mathbb{C}~ |~ f(z)=a,counting ~without~ multiplicity\}.$$
Two meromorphic functions $f$ and $g$ are said to share the set $S$ counting multiplicities(CM), if $E_{f}(S)=E_{g}(S)$. They are said to share $S$ ignoring multiplicities(IM), if $\ol{E}_{f}(S)=\ol{E}_{g}(S).$
\end{defi}
\begin{defi} A set $S\subset \mathbb{C}\bigcup\{\infty\}$ is called a unique range set for meromorphic functions (in short, URSM), if for any two non-constant meromorphic functions $f$ and $g$ the condition $E_{f}(S)=E_{g}(S)$ implies $f\equiv g$.\par
Similarly we can define unique range set for entire functions( URSE).
\end{defi}
\begin{defi} A set $S\subset \mathbb{C}\bigcup\{\infty\}$ is called a unique range set for meromorphic functions  ignoring multiplicities (in short, URSM-IM), if for any two non-constant meromorphic functions $f$ and $g$ the condition $\ol{E}_{f}(S)=\ol{E}_{g}(S)$ implies $f\equiv g$.\par
Similarly we can define unique range set for entire functions ignoring multiplicities ( URSE-IM).
\end{defi}
We further recall the notion of weighted sharing of sets appeared in the literature in 2001 (\cite{9}). As far as relaxations of the nature of sharing of the sets are concerned, this notion has a remarkable influence.
\begin{defi} (\cite{9}) Let $k$ be a nonnegative integer or infinity. For $a\in\mathbb{C}\cup\{\infty\}$ we denote by $E_{k}(a;f)$ the set of all $a$-points of $f$, where an $a$-point of multiplicity $m$ is counted $m$ times if $m\leq k$ and $k+1$ times if $m>k$. If $E_{k}(a;f)=E_{k}(a;g)$, we say that $f,g$ share the value $a$ with weight $k$.\end{defi}

We write $f$, $g$ share $(a,k)$ to mean that $f$, $g$ share the value $a$ with weight $k$. Clearly if $f$, $g$ share $(a,k)$, then $f$, $g$ share $(a,p)$ for any integer $p$, $0\leq p<k$. Also we note that $f$, $g$ share a value $a$ IM or CM if and only if $f$, $g$ share $(a,0)$ or $(a,\infty)$ respectively.\par
\begin{defi}\cite{9} Let $S$ be a set of distinct elements of $\mathbb{C}\cup\{\infty\}$ and $k$ be a nonnegative integer or $\infty$. We denote by $E_{f}(S,k)$,  the set $\bigcup\limits_{a\in S}E_{k}(a;f)$. If $E_{f}(S,k)=E_{g}(S,k)$, then we say $f$, $g$ share the set $S$ with weight $k$. \end{defi}

\begin{defi}(\cite{10})
  A polynomial $P$ in $\mathbb{C}$, is called a uniqueness polynomial for meromorphic (entire) functions, if for any two non-constant meromorphic (entire) functions $f$ and $g$, $P(f)\equiv P(g)$ implies $f\equiv g$. We say $P$ is a UPM (UPE) in brief.
 \end{defi}
 \begin{defi}(\cite{4}, \cite {6})
Let $P(z)$ be a polynomial such that $P^{'}(z)$ has mutually $t$ distinct  zeros given by $d_{1}, d_{2}, \ldots, d_{t}$ with multiplicities $q_{1}, q_{2}, \ldots, q_{t}$ respectively then $P(z)$ is said to satisfy critical injection property if $P(d_i)\not =P(d_j)$ for $i\not=j$ where $i,j\in \{1,2,\cdot\cdot\cdot,t\}$ .
\end{defi}
 From the definition it is obvious that  $P(z)$ is injective on the set of distinct zeros of $P^{'}(z)$  which are known as critical points of $P(z)$. Furthermore any polynomial $P(z)$ satisfying this property is called critically injective polynomial. Thus a critically injective polynomial has at-most one multiple zero.\par
  To this end, we recall two definitions.
 \begin{defi} (\cite{2.1}) Let $z_{0}$ be a zero of $f-a$ of multiplicity $p$ and a zero of $g-a$ of multiplicity $q$. We denote by $\ol N_{L}(r,a;f)$ the counting function of those $a$-points of $f$ and $g$ where $p>q\geq 1$, by $N^{1)}_{E}(r,a;f)$ the counting function of those $a$-points of $f$ and $g$ where $p=q=1$ and by $\ol N^{(2}_{E}(r,a;f)$ the counting function of those $a$-points of $f$ and $g$ where $p=q\geq 2$, each point in these counting functions is counted only once. In the same way we can define $\ol N_{L}(r,a;g),\; N^{1)}_{E}(r,a;g),\; \ol N^{(2}_{E}(r,a;g).$\end{defi}
\begin{defi} (\cite{2.1}) Let $f$, $g$ share a value $a$ IM. We denote by $\ol N_{*}(r,a;f,g)$ the reduced counting function of those $a$-points of $f$ whose multiplicities differ from the multiplicities of the corresponding $a$-points of $g$.\par

Clearly $\ol N_{*}(r,a;f,g)\equiv\ol N_{*}(r,a;g,f)$ and $\ol N_{*}(r,a;f,g)=\ol N_{L}(r,a;f)+\ol N_{L}(r,a;g)$.
\end{defi}
In 1976 Gross(\cite{7}, Question 6) proposed a problem concerning the uniqueness of entire functions that share sets of distinct elements instead of values as follows :\\
{\bf Question A :} {\it Can one find two finite set $S_{j}$ for $j=1,2$  such that any two non-constant entire functions $f$ and $g$ satisfying $E_{f}(S_{j})=E_{g}(S_{j})$ for $j=1,2$  must be identical ?}\par
In (\cite{7}), Gross also asked : \enquote{If the answer to Question 6 is affirmative, it would be interesting
to know how large both sets would have to be.}\par
Yi (\cite{13}) and independently Fang-Xu (\cite{5.1}) gave a positive answer to Question A. In fact, Yi (\cite{13}) proved that the smallest cardinalities of $S_1$ and $S_2$ are $1$ and $3$ respectively, where $S_1$ and $S_2$ are two finite sets such that any two non-constant entire functions
$f$ and $g$ satisfying $E(S_{j},f)= E(S_{j},g)$ for $j = 1, 2$ must be identical. And till today this is the best result. \par
Now it is natural to ask the following question :\par
{\bf Question B :}(\cite{13.1},\cite{13.2},\cite{12.1}) {\it Can one find two finite sets $S_{j}$ $(j=1,2)$ such that any two non-constant meromorphic functions $f$ and $g$ satisfying $E_{f}(S_{j},\infty)=E_{g}(S_{j},\infty)$ for $j=1,2$  must be identical ?}\par

In 1994, Yi (\cite{12.2}) proved that there exist two finite sets $S_1$ (with 2 elements) and $S_2$ (with 9 elements) such that any two
non-constant meromorphic functions $f$ and $g$ satisfying $E_{f}(S_{j},\infty)=E_{g}(S_{j},\infty)$ for $j=1,2$ must be identical.\par
In (\cite{11}), Li-Yang proved that there exist two finite sets $S_1$ (with 1 element) and $S_2$ (with 15 elements) such that any two non-constant meromorphic functions $f$ and $g$ satisfying $E_{f}(S_{j},\infty)=E_{g}(S_{j},\infty)$ for $j=1,2$ must be identical.\par
 In (\cite{5.0}), Fang-Guo proved that there exist two finite sets $S_1$ (with 1 element) and $S_2$ (with 9 elements) such
that any two non-constant meromorphic functions $f$ and $g$ satisfying $E_{f}(S_{j},\infty)=E_{g}(S_{j},\infty)$ for $j=1,2$ must be identical.\par
Also in 2002, Yi (\cite{13.1}) proved that there exist two finite sets $S_1$ (with 1 element) and $S_2$ (with 8 elements) such that any two non-constant meromorphic functions $f$ and $g$ satisfying $E_{f}(S_{j},\infty)=E_{g}(S_{j},\infty)$ for $j=1,2$ must be identical.\par
In 2008, the first author (\cite{2.1}) improved the result of Yi (\cite{13.1}) by relaxing the nature of sharing the range sets by the notion of weighted sharing. He established that there exist two finite sets $S_1$ (with 1 element) and $S_2$ (with 8 elements) such that any two non-constant meromorphic functions $f$ and $g$ satisfying $E_{f}(S_{1},0)=E_{g}(S_{1},0)$ and $E_{f}(S_{2},2)=E_{g}(S_{2},2)$  must be identical.\par
So the natural query would be whether there exists similar types of unique range sets corresponding to the derivatives of two meromorphic functions. But in this particular direction the number of results are scanty.
The following uniqueness results have been obtained when the derivatives of meromorphic functions sharing one or two are studied by the researchers.
\begin{theoA} (\cite{5.2,12.1}) Let $S_{1}=\{z:z^{n}+az^{n-1}+b=0\}$ and $S_{2}=\{\infty\}$, where $a$, $b$ are nonzero constants such that $z^{n}+az^{n-1}+b=0$ has no repeated root and $n\;(\geq 7)$, $k$ be two positive integers. Let $f$ and $g$ be two non-constant meromorphic functions such that $E_{f^{(k)}}(S_{1},\infty)=E_{g^{(k)}}(S_{1},\infty)$ and $E_{f}(S_{2},\infty)=E_{g}(S_{2},\infty)$ then $f^{(k)}\equiv g^{(k)}$.
\end{theoA}
In 2010, Banerjee-Bhattacharjee (\cite{3.2}) improved the above results in the following way :
\begin{theoB} (\cite{3.2})
Let $S_{i}$, $i=1,2$ and $k$ be given as in {\em Theorem A}. Let $f$ and $g$ be two non-constant meromorphic functions  such that $E_{f^{(k)}}(S_{1},2)=E_{g^{(k)}}(S_{1},2)$ and $E_{f}(S_{2},1)=E_{g}(S_{2},1)$ then $f^{(k)}\equiv g^{(k)}$.
\end{theoB}
\begin{theoC}(\cite{3.2})
Let $S_{i}$, $i=1,2$ be given as in {\em Theorem A}. Let $f$ and $g$ be two non-constant meromorphic functions  such that $E_{f^{(k)}}(S_{1},3)=E_{g^{(k)}}(S_{1},3)$ and $E_{f}(S_{2},0)=E_{g}(S_{2},0)$ then $f^{(k)}\equiv g^{(k)}$.
\end{theoC}
In 2011, Banerjee-Bhattacharjee (\cite{3.3}) further improved the above results in the following manner :
\begin{theoD} (\cite{3.3})  Let $S_{i}$, $i=1,2$ and $k$ be given as in {\em Theorem A}. Let $f$ and $g$ be two non-constant meromorphic functions  such that $E_{f^{(k)}}(S_{1},2)=E_{g^{(k)}}(S_{1},2)$ and $E_{f}(S_{2},0)=E_{g}(S_{2},0)$ then $f^{(k)}\equiv g^{(k)}$.
\end{theoD}
So far from the above discussions we see that for the two set sharing problems, the best result has been obtained when one set contain $8$ elements and the other set contain $1$ element. On the other hand, when derivatives of the functions are considered then the cardinality of one set can further be reduced to $7$.
So it will be natural query whether there can be a single result corresponding to uniqueness of the function sharing two sets which can accommodate the derivative counterpart of the main function as well under relaxed sharing hypothesis with smaller cardinalities than the existing results.
This is the motivation of the paper.
Our one result present in the paper improves all the preceding theorems stated so far {\it Theorems A-D} in some sense.
\section{Main Results}
Suppose for two positive integers $m$,$n$ we shall denote by $P(z)$ the following polynomial.
\bea\label{abcp1}
P(z)=z^{n}-\frac{2n}{n-m}z^{n-m}+\frac{n}{n-2m}z^{n-2m}+c,
\eea
where $c$ is any complex number satisfying $|c|\not=\frac{2m^2}{(n-m)(n-2m)}$ and $c\not=0,-\frac{1-\frac{2n}{n-m}+\frac{n}{n-2m}}{2}$.\par
Following theorems are the main results of the paper. In the first theorem we consider the uniqueness of meromorphic functions and its derivatives counterpart corresponding to single set sharing.

\begin{theo}\label{thB3}
Let $n(\geq 1),~m(\geq1),~k(\geq0)$ be three positive integers such that $m,n$ has no common factors. Let $S=\{z :P(z)=0\}$ where the polynomial $P(z)$ defined by  \ref{abcp1}. Let $f$ and $g$ be two non-constant meromorphic function satisfying $E_{f^{(k)}}(S,l)=E_{g^{(k)}}(S,l)$. If one of the following conditions holds:
\begin{enumerate}
\item $l\geq2$ and $n>\max\{ 2m+4+\frac{4}{k+1},4m+1\}$,
\item $1= l$ and $n> \max\{2m+4.5+\frac{4.5}{k+1},4m+1\}$,
\item $l=0$ and  $n> \max\{2m+7+\frac{7}{k+1},4m+1\}$
\end{enumerate}
then  $f^{(k)} \equiv g^{(k)}$.
\end{theo}
\begin{cor}Let $n(\geq 9),~m(=1),~k\geq1$ be three positive integers. Let $S=\{z :P(z)=0\}$ where the polynomial $P(z)$ defined by  \ref{abcp1}. Let $f$ and $g$ be two non-constant meromorphic function satisfying $E_{f^{(k)}}(S,2)=E_{g^{(k)}}(S,2)$. Then $f^{(k)} \equiv g^{(k)}$.
\end{cor}
For $k=0$ in Theorem \ref{thB3} we get the following :
\begin{cor}Let $n(\geq 1),~m(\geq1)$ be two positive integers having no common factors. Let $S=\{z :P(z)=0\}$ where the polynomial $P(z)$ defined by  \ref{abcp1}. Let $f$ and $g$ be two non-constant meromorphic function satisfying $E_{f}(S,l)=E_{g}(S,l)$. If one of the following conditions holds:
\begin{enumerate}
\item $l\geq2$ and $n>\max\{2m+8,4m+1\}$,
\item $1= l$ and $n>\max\{2m+9,4m+1\}$,
\item $l=0$ and  $n>\max\{2m+14,4m+1\}$
\end{enumerate}
then  $f \equiv g$.
\end{cor}
The next theorem focus on the two set sharing problem.
\begin{theo}\label{thB5}
Let $n(> 4m+1),~m(\geq1),~k(\geq0)$ be three positive integers satisfying $\gcd\{m,n\}=1$. Let $S=\{z :P(z)=0\}$ where the polynomial $P(z)$ defined by  \ref{abcp1}. Let $f$ and $g$ be two non-constant meromorphic function satisfying $E_{f^{(k)}}(S,l)=E_{g^{(k)}}(S,l)$ and $E_{f^{(k)}}(0,q)=E_{g^{(k)}}(0,q)$ where $0\leq q < \infty$. If
\begin{enumerate}
\item $l\geq \frac{3}{2}+\frac{2}{n-2m-1}+\frac{1}{(n-2m)q+n-2m-1}$ and
\item  $n>2m+\frac{4}{k+1}+\frac{4}{(k+1)(n-2m-1)}+\frac{2}{(k+1)((n-2m)q+n-2m-1)}$
\end{enumerate}
then $f^{(k)}\equiv g^{(k)}$.
\end{theo}
The following  example shows that for the two set sharing case,  choosing the set $S_{1}$ with one element and $S_{2}$ with two elements Theorem \ref{thB5} ceases to hold.
\begin {exm}\label{abc2015.}
Let $S_{1}=\{a\}$ and $S_{2}=\{b,c\}$. Choose $f(z) =p(z)+(b-a)e^{z}$ and $g(z)=q(z)+(-1)^{k}(c-a)e^{-z}$, where $p(z)$ and $q(z)$ are polynomial of degree $k$ with the coefficient of $z^{k}$ in $p(z)$ and $q(z)$ is equal to $\frac{a}{k!}$.\par
Clearly $E_{f^{(k)}}(S_j)=E_{g^{(k)}}(S_j)$ for $j=1,2$  but $f^{(k)}\not\equiv g^{(k)}$.
\end{exm}
\begin{cor}Let $n(\geq 8),~m(=1)$ be two positive integers. Let $S=\{z :P(z)=0\}$ where the polynomial $P(z)$ defined by  \ref{abcp1}. Let $f$ and $g$ be two non-constant meromorphic function satisfying
\begin{enumerate}
\item $E_{f}(S,3)=E_{g}(S,3)$ and $E_{f}(0,0)=E_{g}(0,0)$, or
\item $E_{f}(S,2)=E_{g}(S,2)$ and $E_{f}(0,1)=E_{g}(0,1)$
\end{enumerate}
then  $f \equiv g$.
\end{cor}
\begin{cor}\label{abe1.4} Let $n(\geq 6),~m(=1)$ and $k\geq1$ be two positive integers. Let $S=\{z :P(z)=0\}$ where the polynomial $P(z)$ defined by  \ref{abcp1}. Let $f$ and $g$ be two non-constant meromorphic function satisfying $E_{f^{(k)}}(S,3)=E_{g^{(k)}}(S,3)$ and $E_{f^{(k)}}(0,0)=E_{g^{(k)}}(0,0)$. Then $f^{(k)} \equiv g^{(k)}$.
\end{cor}
\begin{rem} Corollary \ref{abe1.4} shows that there exists two sets $S_{1}$ (with 1 element) and $S_{2}$ (with $6$ elements) such that when derivatives of any two non-constant meromorphic functions share them with finite weight yields $f^{(k)}\equiv g^{(k)}$ thus improve {\it Theorem D} in the direction of {\it Question B}.\end{rem}

The following two examples show that specific form of choosing the set $S_{1}$ with five elements and $S_{2}=\{0\}$ Corollary \ref{abe1.4} ceases to hold.
\begin {exm}\label{abc2015}
Let $f(z) =\frac{1}{(\sqrt{\alpha \beta \gamma })^{k-1}}e^{\sqrt{\alpha \beta \gamma}\;z}$ and $g(z)=\frac{(-1)^{k}}{(\sqrt{\alpha \beta \gamma})^{k-1}}e^{-\sqrt{\alpha \beta \gamma}\;z}$ ($k\geq 1$) and \linebreak $S=\{\alpha \sqrt{\beta },\alpha \sqrt{\gamma},\beta \sqrt{\gamma },\gamma \sqrt{\beta },\sqrt(\alpha\beta\gamma)\}$, where $\alpha $, $\beta $ and $\gamma $ are three nonzero distinct complex numbers. Clearly $E_{f^{(k)}}(S)=E_{g^{(k)}}(S)$ and $E_{f^{(k)}}(0)=E_{g^{(k)}}(0)$  but $f^{(k)}\not\equiv g^{(k)}$.
\end{exm}
\begin{exm} Let $f(z)=\frac{1}{c^{k}}e^{cz}$ and $g(z)=\omega^{4}f(z)$ and $S=\{\omega^{4},\omega^{3},\omega^{2},\omega,1\}$, where $\omega$ is the non-real fifth root of unity and $c$ is a non-zero complex number. Clearly $E_{f^{(k)}}(S)=E_{g^{(k)}}(S)$ and $E_{f^{(k)}}(0)=E_{g^{(k)}}(0)$  but $f^{(k)}\not\equiv g^{(k)}$.
\end{exm}
\begin{rem} However the following question is still inevitable from the Corollary \ref{abe1.4} and Example \ref{abc2015}  that \par
 \emph{Whether there exists two suitable sets $S_{1}$ (with 1 element) and $S_{2}$ (with $5$ elements) such that when derivatives of any two non-constant meromorphic functions share them with finite weight yield $f^{(k)}\equiv g^{(k)}$ ?}\end{rem}

\section {Lemmmas}

We define\\
 $F=-\frac{(f^{(k)})^{n-2m}((f^{(k)})^{2m}-\frac{2n}{n-m}(f^{(k)})^{m}+\frac{n}{n-2m})}{c}$ , $G=-\frac{(g^{(k)})^{n-2m}((g^{(k)})^{2m}-\frac{2n}{n-m}(g^{(k)})^{m}+\frac{n}{n-2m})}{c},$\\
 where $n(\geq 1)$, $m(\geq 1)$ and $k(\geq 0)$ are non-negative integers.
Henceforth we shall denote by $H$ and $\Phi$ the following two functions $$H=(\frac{F^{''}}{F^{'}}-\frac{2F^{'}}{F-1})-(\frac{G^{''}}{G^{'}}-\frac{2G^{'}}{G-1})$$ and $$\Phi=\frac{F^{'}}{F-1}-\frac{G^{'}}{G-1}$$
Let $T(r)=\max\{T(r,f^{(k)}),T(r,g^{(k)})\}$ and $S(r)=o(T(r))$.\par
\begin{lem}\label{b0}  The polynomial $$P(z)=z^{n}-\frac{2n}{n-m}z^{n-m}+\frac{n}{n-2m}z^{n-2m}+c $$ is a critically injective polynomial having only simple zeros when $|c|\not=0,\frac{2m^2}{(n-m)(n-2m)}$.
\end{lem}
\begin{proof} Since $$P'(z)=nz^{n-2m-1}(z^{m}-1)^{2}$$
$P$ is critically injective, because
\begin{enumerate}
\item $P(0)=P(\alpha)$ where $\alpha^{m}=1$ gives $\alpha=0$ which is a contradiction, and
\item $P(\beta)=P(\gamma)$ where $\beta^{m}=1,\gamma^{m}=1$, gives $\beta^{n}=\gamma^{n}$.\\
Now as $\gcd\{m,n\}=1$, so there exist integers $s,t$ such that $ms+nt=1$.\\
Thus $\beta=\beta^{ms+nt}=\gamma^{ms+nt}=\gamma$.
\end{enumerate}
For the second part if $P(\alpha)=P'(\alpha)=0$ then either $\alpha=0$ or $\alpha^{m}=1$.\\
If $\alpha=0$ then $P(\alpha)=c$, which is not zero by assumption.\\
If  $\alpha^{m}=1$ then $P(\alpha)=\alpha^{n}(1-\frac{2n}{n-m}+\frac{n}{n-2m})+c=0$\\
As $|\alpha|=1$ so $|c|=\frac{2m^{2}}{(n-m)(n-2m)}$ which is not possible by assumption.
\end{proof}

\begin{lem}\label{b1}(\cite{6})  Suppose that $P(z)$ is a monic polynomial without multiple zero whose derivatives has mutually distinct $t$ zeros given by $d_{1}, d_{2}, \ldots, d_{t}$ with multiplicities $q_{1}, q_{2}, \ldots, q_{t}$ respectively. Also suppose that $P(z)$ is critically injective.  Then $P(z)$ will be a uniqueness polynomial if and only if $$\sum \limits_{1\leq l<m\leq k}q_{_{l}}q_{m}>\sum \limits_{l=1}^{t} q_{_{l}}.$$
In particular the above inequality is always satisfied whenever $t\geq 4$. When $t=3$ and $\max \{q_{1},q_{2},q_{3}\}\geq 2$ or when $t=2$, $\min \{q_{1},q_{2}\}\geq 2$ and $q_{1}+q_{2}\geq 5$ then also the above inequality holds.\
\end{lem}
\begin{lem}\label{bd2}  $F$ and $G$ are defined as earlier. Then $F\equiv G$ gives $f^{(k)}\equiv g^{(k)}$ when $k\geq0$ and $n\geq 2m+4$.
\end{lem}
\begin{proof} $F\equiv G$ implies $P(f^{(k)})=P(g^{(k)})$.\\
Since $P$ is critically injective polynomial having no multiple zeros and $$P'(z)=nz^{n-2m-1}(z^{m}-1)^{2},~t=m+1.$$ So
 when $n\geq 2m+4$ we have by the Lemma \ref{b1} that $f^{(k)}\equiv g^{(k)}$.
\end{proof}
\begin{lem}\label{bb3} $F$ and $G$ are defined as earlier, then $FG\not\equiv 1$ for $k\geq0$ and $n\geq5$.
\end{lem}
\begin{proof} On contrary, suppose $FG\equiv 1$\\
Then by  Mokhon'ko's Lemma(\cite{11.1}), $T(r,f^{(k)})=T(r,g^{(k)})+O(1)$.\\
Then \bea\label{r0}(f^{(k)})^{n-2m}\prod_{i=1}^{2m}((f^{(k)})-\gamma_{i})(g^{(k)})^{n-2m}\prod_{i=1}^{2m}((g^{(k)})-\gamma_{i})=c^{2},\eea
where $\gamma_{i}$ (i=1,2,...,2m) are the roots of the equation $z^{2m}-\frac{2n}{n-m}z^{m}+\frac{n}{n-2m}=0$.\par
Let $z_{0}$ be a $\gamma_{i}$ point of $f^{(k)}$ of order $p$. Then $z_{0}$ is a pole of $g$ of order $q$ such that $p=n(1+k)q\geq n$.
So $$\ol{N}(r,\gamma_{i};f^{(k)})\leq\frac{1}{n}N(r,\gamma_{i};f^{(k)}).$$
Again let $z_{0}$ be a zero of $f^{(k)}$ of order $t$. Then $z_{0}$ is a pole of $g$ of order $s$ such that $(n-2m)t=ns(1+k)$.\\
Thus $t>s(1+k)$ and $2ms(1+k)=(n-2m)(t-s(1+k))\geq(n-2m)$. Consequently $(n-2m)t=ns(1+k)$ gives $t\geq\frac{n}{2m}.$
So $$\ol{N}(r,0;f^{(k)})\leq\frac{2m}{n}N(r,0;f^{(k)}).$$
again \beas \ol{N}(r,\infty;f^{(k)}) &\leq& \ol{N}(r,0;g^{(k)})+\sum\limits_{i=0}^{2m}\ol{N}(r,\gamma_{i};g^{(k)})\\
 &\leq& \frac{2m}{n}N(r,0;g^{(k)})+\frac{1}{n}\sum\limits_{i=0}^{2m}N(r,\gamma_{i};g^{(k)})\\
 &\leq& \frac{4m}{n}T(r,g^{(k)})+O(1).\eeas
 Now by using the Second Fundamental Theorem we get
 \bea\label{r1} && 2mT(r,f^{(k)})\\
\nonumber  &\leq& \ol{N}(r,\infty;f^{(k)})+\ol{N}(r,0;f^{(k)})+\sum\limits_{i=0}^{2m}\ol{N}(r,\gamma_{i};f^{(k)})+S(r,f^{(k)})\\
\nonumber  &\leq& \frac{4m}{n}T(r,f^{(k)})+\frac{2m}{n}T(r,f^{(k)})+\frac{2m}{n}T(r,f^{(k)})+S(r,f^{(k)}),\eea
  which is a contradiction as $n\geq5$.
\end{proof}
\begin{lem}\label{b4}(\cite{3.1})
If $F$ and $G$ share $(1,l)$ where $0\leq l<\infty$ then\\
$\ol{N}(r,1;F)+\ol{N}(r,1;G)-N_{E}^{1}(r,1,F)+(l-\frac{1}{2})\ol{N}_{*}(r,1;F,G)\leq\frac{1}{2}(N(r,1;F)+N(r,1;G)).$
\end{lem}

\begin{lem}\label{b7} Let $F$, $G$, $\Phi$ be defined previously and $F\not\equiv G$. If $f^{(k)}$ and $g^{(k)}$ share $(0,q)$ where $0\leq q<\infty$ and $F$, $G$ share $(1,l)$ then
\beas & &\{(n-2m)q+n-2m-1\}\;\ol N(r,0;f^{(k)}\mid\geq q+1)\\
&\leq& \ol{N}(r,\infty;f^{(k)})+\ol{N}(r,\infty;g^{(k)})+\ol{N}_{*}(r,1;F,G)+S(r).\eeas
Similar expressions hold for $g$ also. \end{lem}
\begin{proof}
Case-1 $\Phi=0$\\
Then by integration we get $$F-1=A(G-1)$$
where $A$ is non-zero constant. Since $F\not\equiv G$. we have $A\not= 1$.Thus $0$ is an e.v.P. of $f^{(k)}$ and $g^{(k)}$ and hence the lemma  follows immediately.\\
Case-2 $\Phi\not=0$\\
Let $z_{0}$ be a zero of $f^{(k)}$ of order $t(\geq q+1)$. Then it is a zero of $F$ of order atleast $(q+1)(n-2m)$ and hence $z_{0}$ is the zero of $\Phi$ of order at least $q(n-2m)+n-2m-1$. Thus\\
\beas & &\{(n-2m)q+n-2m-1\}\;\ol N(r,0;f^{(k)}\mid\geq q+1)\\
&\leq& N(r,0;\Phi)\\
&\leq& T(r,\Phi)+O(1)\\
&\leq& N(r,\infty;\Phi)+S(r)\\
&\leq& \ol{N}(r,\infty;f^{(k)})+\ol{N}(r,\infty;g^{(k)})+\ol{N}_{*}(r,1;F,G)+S(r).\eeas
\end{proof}
\begin{lem}\label{biku} Let $H$ be defined previously. If $H\equiv 0$ and $n\geq 4m+2$ with $\gcd\{m,n\}=1$ then $f^{(k)}\equiv g^{(k)}$ for any integer $k\geq0$.
\end{lem}
\begin{proof}
In this case $F$ and $G$ share $(1,\infty)$.\par
Now by integration we have
\bea\label{pe1.1} F=\frac{AG+B}{CG+D},\eea
where $A,B,C,D$ are constant satisfying $AD-BC\neq 0 $.\par
Thus by Mokhon'ko's Lemma (\cite{11.1}) \bea\label{pe1.2} T(r,f^{(k)})=T(r,g^{(k)})+S(r).\eea

As $AD-BC\neq0$, so $A=C=0$ never occur. Thus we consider the following cases:\\
\textbf{Case-1} $AC\neq0$\\
In this case \bea F-\frac{A}{C}=\frac{BC-AD}{C(CG+D)}.\eea
So, $$\overline{N}(r,\frac{A}{C};F)=\overline{N}(r,\infty;G).$$
Now by using the Second Fundamental Theorem and (\ref{pe1.2}), we get
\beas && n T(r,f^{(k)})+O(1)=T(r,F)\\
 &\leq& \overline{N}(r,\infty;F)+\overline{N}(r,0;F)+\overline{N}(r,\frac{A}{C};F)+S(r,F)\\
&\leq& \overline{N}(r,\infty;f^{(k)})+\overline{N}(r,0;f^{(k)})+2mT(r,f^{(k)})+\overline{N}(r,\infty;g^{(k)})+S(r,f^{(k)})\\
&\leq& (2m+1+\frac{2}{k+1})T(r,f^{(k)})+S(r,f^{(k)}),\eeas
which is a contradiction as $n\geq 4m+2$.\\

\textbf{Case-2} $AC=0$\\
\textbf{Subcase-2.1} $A=0$ and $C\neq0$\\
In this case $B\neq0$ and
$$F=\frac{1}{\gamma G+\delta},$$
where $\gamma=\frac{C}{B}$ and $\delta=\frac{D}{B}$.\\
If $F$ has no $1$-point,  then by using the Second Fundamental Theorem and (\ref{pe1.2}),\\ we get
\beas &&T(r,F)\\
 &\leq& \overline{N}(r,\infty;F)+\overline{N}(r,0;F)+\overline{N}(r,1;F)+S(r,F)\\
&\leq& \overline{N}(r,\infty;f^{(k)})+\overline{N}(r,0;f^{(k)})+2mT(r,f^{(k)})+S(r,f^{(k)})\\
&\leq& \frac{2m+1+\frac{1}{k+1}}{n}T(r,F)+S(r,F),\eeas
which is a contradiction as $n\geq 4m+2$.\par
Thus $\gamma+\delta=1$ and $\gamma\neq0$.\par
So, $$F=\frac{1}{\gamma G+1-\gamma},$$
From above we get $\overline{N}(r,0;G+\frac{1-\gamma}{\gamma})=\overline{N}(r,\infty;F)$.\\
If $\gamma\neq1$, by using the Second Fundamental Theorem and (\ref{pe1.2}), we get
\beas && T(r,G)\\
 &\leq& \overline{N}(r,\infty;G)+\overline{N}(r,0;G)+\overline{N}(r,0;G+\frac{1-\gamma}{\gamma})+S(r,G)\\
&\leq& \overline{N}(r,\infty;g^{(k)})+\overline{N}(r,0;g^{(k)})+2mT(r,g^{(k)})+\overline{N}(r,\infty;f^{(k)})+S(r,g^{(k)})\\
&\leq& \frac{2m+1+\frac{2}{k+1}}{n}T(r,F)+S(r,F),\eeas
which is a contradiction as $n\geq 4m+2$.\par

Thus $\gamma=1$ and $FG\equiv 1$ which is not possible by Lemma \ref{bb3}.\\
\textbf{Subcase-2.2} $A\neq0$ and $C=0$\par
In this case $D\neq0$ and
$$F=\lambda G+\mu,$$
where $\lambda=\frac{A}{C}$ and $\mu=\frac{B}{D}$.\par
If $F$ has no $1$ point then similarly as above we get a contradiction.\par
Thus $\lambda+\mu=1$ with $\lambda\neq0$.\par
Clearly $\overline{N}(r,0;G+\frac{1-\lambda}{\lambda})=\overline{N}(r,0;F)$.\par
Let $\lambda \neq1$ and $\xi=\frac{(1-\frac{2n}{n-m}+\frac{n}{n-2m})}{c}$.
Then $F+\xi=(f^{(k)}-1)^{3}Q_{n-3}(f^{(k)})$,\\ where $Q_{n-3}(1)\not= 0$ and $Q_{n-3}(z)$ is a $(n-3)$ degree polynomial.\par
If $\frac{1-\lambda}{\lambda}\not=\xi$, then by using the Second Fundamental Theorem and (\ref{pe1.2}), we get
\beas &&2T(r,G)\\
 &\leq& \overline{N}(r,\infty;G)+\overline{N}(r,0;G)+\overline{N}(r,0;G+\frac{1-\lambda}{\lambda})+\overline{N}(r,0;G+\xi)+S(r,G)\\
&\leq& \overline{N}(r,\infty;g^{(k)})+\overline{N}(r,0;g^{(k)})+2mT(r,g^{(k)})+\overline{N}(r,0;f^{(k)})+2mT(r,f^{(k)})\\
&+& \overline{N}(r,1;g^{(k)})+(n-3)T(r,g^{(k)})+S(r,g^{(k)})\\
&\leq& \frac{4m+n+\frac{1}{k+1}}{n}T(r,G)+S(r,G),\eeas
which is a contradiction as $n\geq 4m+2$.\par
If $\frac{1-\lambda}{\lambda}=\xi$, then $\lambda G=F-\lambda\xi$. As $c\not=-\frac{1-\frac{2n}{n-m}+\frac{n}{n-2m}}{2}$ so $\lambda\not=-1$.\par
Now applying the Second Fundamental Theorem and (\ref{pe1.2}), we get
\beas && 2T(r,F)\\
 &\leq& \overline{N}(r,\infty;F)+\overline{N}(r,0;F)+\overline{N}(r,0;F-\lambda\xi)+\overline{N}(r,0;F+\xi)+S(r,F)\\
&\leq& \overline{N}(r,\infty;f^{(k)})+\overline{N}(r,0;g^{(k)})+2mT(r,g^{(k)})+\overline{N}(r,0;f^{(k)})+2mT(r,f^{(k)})\\
&+& \overline{N}(r,1;f^{(k)})+(n-3)T(r,f^{(k)})+S(r,g^{(k)})\\
&\leq& \frac{4m+n+\frac{1}{k+1}}{n}T(r,F)+S(r,F),\eeas
which is a contradiction as $n > 4m+1$.\par
Thus $\lambda=1$ and $F\equiv G$.
Consequently by Lemma \ref{bd2}, $f^{(k)}\equiv g^{(k)}$ .
\end{proof}
\section {Proof of the theorems}

\begin{proof} [\textbf{Proof of Theorem\ref{thB3} }]
It is clear that $\overline{N}(r,\infty;f^{(k)})\leq\frac{1}{k+1}N(r,\infty;f^{(k)})$.\\
\textbf{Case-1} $H \not\equiv 0$\\
Clearly $F'=-\frac{n}{c}(f^{(k)})^{n-2m-1}((f^{(k)})^{m}-1)^{2}(f^{(k+1)})$, \\and  $G'=-\frac{n}{c}(g^{(k)})^{n-2m-1}((g^{(k)})^{m}-1)^{2}(g^{(k+1)})$.
\\Now by simple calculations,
\bea\nonumber\label{cb1} && N(r,\infty;H)\\
\nonumber &\leq& \ol{N}(r,0;F|\geq2)+\ol{N}(r,0;G|\geq2)+\ol{N}(r,\infty;F)\\
\nonumber&+& \ol{N}(r,\infty;G)+\ol{N}_{*}(r,1;F,G)+\ol{N}_{0}(r,0;F')+\ol{N}_{0}(r,0;G'),\eea
where $\ol{N}_{0}(r,0;F')$ is the reduced counting function of zeros of $F'$ which is not zeros of $F(F-1)$.
Thus
\bea\label{cb2} &&N(r,\infty;H)\\
\nonumber &\leq& \ol{N}(r,0;f^{(k)})+\ol{N}(r,0;g^{(k)})+\ol{N}(r,0;((g^{(k)})^{m}-1))\\
\nonumber&+&\ol{N}(r,0;((f^{(k)})^{m}-1))+\ol{N}(r,\infty;f^{(k)})+\ol{N}(r,\infty;g^{(k)})\\
\nonumber&+&\ol{N}_{*}(r,1;F,G)+\ol{N}_{0}(r,0;f^{(k+1)})+\ol{N}_{0}(r,0;g^{(k+1)}),\eea
where $\ol{N}_{0}(r,0;f^{(k+1)})$ is the reduced counting function of zeros of $f^{(k+1)}$ which is not zeros of $f^{(k)}((f^{(k)})^m-1)$ and $(F-1)$.

Clearly \bea \label{cb2.5} \overline{N}(r,1;F|=1)=\overline{N}(r,1;G|=1)\leq N(r,\infty;H).\eea
Now by using the Second Fundamental Theorem, (\ref{cb2}), (\ref{cb2.5}) and Lemma \ref{b4} we get\\
\bea\label{n2}&& (n+m)(T(r,f^{(k)})+T(r,g^{(k)}))\\
\nonumber &\leq& \overline{N}(r,\infty;f^{(k)})+\overline{N}(r,0;f^{(k)})+\overline{N}(r,\infty;g^{(k)})+\overline{N}(r,0;g^{(k)})\\
\nonumber &+& \overline{N}(r,1;F)+\overline{N}(r,1;G)+\ol{N}(r,0;(f^{(k)})^{m}-1)+\ol{N}(r,0;(g^{(k)})^{m}-1)\\
\nonumber&-& N_{0}(r,0,f^{(k+1)})-N_{0}(r,0,g^{(k+1)})+S(r,f^{(k)})+S(r,g^{(k)})\\
\nonumber &\leq& 2\{\overline{N}(r,\infty;f^{(k)})+\overline{N}(r,\infty;g^{(k)})\}+2\{\overline{N}(r,0;f^{(k)})+\overline{N}(r,0;g^{(k)})\\
\nonumber &+& \ol{N}(r,0;((g^{(k)})^{m}-1))+\ol{N}(r,0;((f^{(k)})^{m}-1))\}+\overline{N}(r,1;F)+\overline{N}(r,1;G)\\
\nonumber &-&\overline{N}(r,1;F|=1)+\ol{N}_{*}(r,1;F,G)+S(r,f^{(k)})+S(r,g^{(k)}).
\eea
\bea\label{n3}&& (\frac{n}{2}-m)(T(r,f^{(k)})+T(r,g^{(k)}))\\
\nonumber &\leq& 2\{\overline{N}(r,\infty;f^{(k)})+\overline{N}_{*}(r,\infty;g^{(k)})+\overline{N}(r,0;f^{(k)})+\overline{N}(r,0;g^{(k)})\}\\
\nonumber &+&(\frac{3}{2}-l)\ol{N}_{*}(r,1;F,G)+S(r,f^{(k)})+S(r,g^{(k)}).
\eea
That is
\bea\label{n3.51}&& (\frac{n}{2}-m-2-\frac{2}{k+1})(T(r,f^{(k)})+T(r,g^{(k)}))\\
\nonumber &\leq& (\frac{3}{2}-l)\ol{N}_{*}(r,1;F,G)+S(r,f^{(k)})+S(r,g^{(k)}).
\eea
\textbf{Subcase-1.1} $l\geq2$\\
We get a contradiction from (\ref{n3.51}) when $n>2m+4+\frac{4}{k+1}$.\\

\textbf{Subcase-1.2} $l=1$\\

In this case
\beas && N_{*}(r,1;F,G)=\ol{N}_{L}(r,1;F)+\ol{N}_{L}(r,1;G)\\
&\leq& \frac{1}{2}(N(r,0;f^{(k+1)}|f^{(k)}\neq0)+N(r,0;g^{(k+1)}|g^{(k)}\neq0))\\
&\leq& \frac{1}{2}(\overline{N}(r,\infty;f^{(k)})+\overline{N}(r,0;f^{(k)})+\overline{N}(r,\infty;g^{(k)})+\overline{N}(r,0;g^{(k)}))\\
&+& S(r,f^{(k)})+S(r,g^{(k)})\\
&\leq& \frac{1}{2}(1+\frac{1}{k+1})(T(r,f^{(k)})+T(r,g^{(k)}))+S(r,f^{(k)})+S(r,g^{(k)}).
\eeas
Thus (\ref{n3.51}) becomes
\bea\label{n3.5}&& (\frac{n}{2}-m-2-\frac{2}{k+1})(T(r,f^{(k)})+T(r,g^{(k)}))\\
\nonumber &\leq&  \frac{1}{4}(1+\frac{1}{k+1})(T(r,f^{(k)})+T(r,g^{(k)}))+S(r,f^{(k)})+S(r,g^{(k)}),
\eea
which is a contradiction when $n>2m+4.5+\frac{4.5}{k+1}$.\\
\textbf{Subcase-1.3} $l=0$\\
In this case \beas && N_{*}(r,1;F,G)=\ol{N}_{L}(r,1;F)+\ol{N}_{L}(r,1;G)\\
&\leq& (N(r,0;f^{(k+1)}|f^{(k)}\neq0)+N(r,0;g^{(k+1)}|g^{(k)}\neq0))\\
&\leq& (\overline{N}(r,\infty;f^{(k)})+\overline{N}(r,0;f^{(k)})+\overline{N}(r,\infty;g^{(k)})+\overline{N}(r,0;g^{(k)}))\\
&+& S(r,f^{(k)})+S(r,g^{(k)})\\
&\leq& (\overline{N}(r,\infty;f^{(k)})+\overline{N}(r,0;f^{(k)})+\overline{N}(r,\infty;g^{(k)})+\overline{N}(r,0;g^{(k)}))\\
&+& S(r,f^{(k)})+S(r,g^{(k)})\\
&\leq& (1+\frac{1}{k+1})(T(r,f^{(k)})+T(r,g^{(k)}))+S(r,f^{(k)})+S(r,g^{(k)}).
\eeas
Thus (\ref{n3.51}) becomes
\bea\label{n3.5}&& (\frac{n}{2}-m-2-\frac{2}{k+1})(T(r,f^{(k)})+T(r,g^{(k)}))\\
\nonumber &\leq&  \frac{3}{2}(1+\frac{1}{k+1})(T(r,f^{(k)})+T(r,g^{(k)}))+S(r,f^{(k)})+S(r,g^{(k)}),
\eea
which is a contradiction when $n>2m+7+\frac{7}{k+1}$.\\
\textbf{Case-2} $H\equiv 0$\\
From the Lemma \ref{biku} we obtained $f^{(k)}\equiv g^{(k)}$ when $n\geq 4m+2$.
\end{proof}

\begin{proof} [\textbf{Proof of Theorem\ref{thB5} }]
\textbf{Case-1} $H \not\equiv 0$\\
Then clearly $F\not\equiv G$.\\
As $f$ and $g$ share $(0,q)$, we have\\
\bea\label{m1} && N(r,\infty;H)\\
\nonumber &\leq& \ol{N}(r,\infty;f^{(k)})+\ol{N}(r,\infty;g^{(k)})+\ol{N}(r,0;((g^{(k)})^{m}-1))\\
\nonumber &+&\ol{N}(r,0;((f^{(k)})^{m}-1))+\ol{N}_{*}(r,0;f^{(k)},g^{(k)})\\
\nonumber &+&\ol{N}_{*}(r,1;F,G)+\ol{N}_{0}(r,0;f^{(k+1)})+\ol{N}_{0}(r,0;g^{(k+1)}),\eea
where $\ol{N}_{0}(r,0;f^{(k+1)})$ is the reduced counting function of zeros of $f^{(k+1)}$ which is not zeros of $f^{(k)}((f^{(k)})^m-1)$ and $(F-1)$.\\
Now using the Second Fundamental Theorem, (\ref{cb2.5}), (\ref{m1}) and Lemma \ref{b4} we get\\
\bea\label{m3}&& (\frac{n}{2}-m)(T(r,f^{(k)})+T(r,g^{(k)}))\\
\nonumber &\leq& 2\overline{N}(r,0;f^{(k)})+\overline{N}_{*}(r,0;f^{(k)},g^{(k)})+2\{\overline{N}(r,\infty;f^{(k)})+\overline{N}(r,\infty;g^{(k)})\}\\
\nonumber &+&(\frac{3}{2}-l)\ol{N}_{*}(r,1;F,G)+S(r,f^{(k)})+S(r,g^{(k)})\\
\nonumber &\leq& 2\overline{N}(r,0;f^{(k)})+\overline{N}(r,0;f^{(k)}|\geq q+1)+2\{\overline{N}(r,\infty;f^{(k)})+\overline{N}(r,\infty;g^{(k)})\}\\
\nonumber &+&(\frac{3}{2}-l)\ol{N}_{*}(r,1;F,G)+S(r,f^{(k)})+S(r,g^{(k)}).
\eea
Thus by the help of Lemma \ref{b7} we have\\
\bea\label{m4} && (\frac{n}{2}-m-\frac{2}{k+1})(T(r,f^{(k)})+T(r,g^{(k)}))\\
\nonumber &\leq& 2\overline{N}(r,0;f^{(k)})+\overline{N}(r,0;f^{(k)}|\geq q+1)+(\frac{3}{2}-l)\ol{N}_{*}(r,1;F,G)\\
\nonumber &+& S(r,f^{(k)})+S(r,g^{(k)})\\
\nonumber &\leq& (\frac{2}{(k+1)(n-2m-1)}+\frac{1}{(k+1)((n-2m)q+n-2m-1)})\{T(r,f^{(k)})\\
\nonumber &+& T(r,g^{(k)})\}+(\frac{2}{n-2m-1}+\frac{1}{(n-2m)q+n-2m-1}+\frac{3}{2}-l)\ol{N}_{*}(r,1;F,G)\\
\nonumber &+& S(r,f^{(k)})+S(r,g^{(k)}).
\eea
Thus when $l\geq \frac{3}{2}+\frac{2}{n-2m-1}+\frac{1}{(n-2m)q+n-2m-1}, $\\
and  $n>2m+\frac{4}{k+1}+\frac{4}{(k+1)(n-2m-1)}+\frac{2}{(k+1)((n-2m)q+n-2m-1)}$,\\
we get a contradiction from (\ref{m4}).\\

\textbf{Case-2} $H\equiv 0$\\
From the Lemma \ref{biku} we obtained $f^{(k)}\equiv g^{(k)}$ when $n\geq 4m+2$.
\end{proof}
\section{\textbf{Acknowledgement}} We would like to thank the editor and the referees.

\end{document}